\documentclass[11pt]{amsart}

\usepackage{amsmath,amssymb,amsthm,mathtools}
\usepackage{mathrsfs}
\usepackage{enumitem}
\usepackage{geometry}
\usepackage{microtype}
\usepackage{hyperref}
\usepackage[pagewise]{lineno}
\usepackage{titlesec}
\titleformat{\section}{\normalfont\Large\bfseries\centering}{\thesection}{1em}{}

\numberwithin{equation}{section}

\newtheorem{definition}{Definition}[section]
\newtheorem{theorem}{Theorem}[section]
\newtheorem{proposition}[theorem]{Proposition}
\newtheorem{lemma}[theorem]{Lemma}

\theoremstyle{remark}

\newcommand{\R}{\mathbb{R}}

\newcommand{\dd}{\,\mathrm{d}}
\newcommand{\sgn}{\operatorname{sgn}}

\newcommand{\Pm}{\mathcal{P}_m}

\title[Limit cycles of discontinuous piecewise linear Li\'enard systems]{Bounds on the maximum number of limit cycles\\
of piecewise linear Li\'enard systems\\[3pt]
II. The discontinuous case}
\author[Chen et al.]{ Hebai Chen$^1$, Jie Jin$^1$, Shu Li$^1$, Yuhuan Lu$^1$}
\date{}

\address{  $^1$ School of Mathematics and Statistics, HNP-LAMA, Central South University,
	Changsha, Hunan 410083, P. R. China}

\email{chen\_hebai@csu.edu.cn}
\email{jin\_jie@csu.edu.cn}
\email{shli97@csu.edu.cn}
\email{lu\_yuhuan@csu.edu.cn}

\subjclass[2020]{34C07, 34A36, 37C20, 37G15}

\keywords{Discontinuous piecewise linear Li\'enard system;
limit cycle; Tonnelier's conjecture;
multiscale perturbation; Pfaffian function}

\begin{document}

	\begin{abstract}
		This paper concerns the planar Li\'enard system
\(\dot x=y-F(x),\ \dot y=-x\), where \(F(x)\) is a
piecewise linear function with exactly \(m\) jump points
and no fold points. Tonnelier [SIAM J. Appl. Math. 63 (2002)]
conjectured that the maximum number of limit cycles of
the system is \(2m\). The conjecture was confirmed for
\(m=1\) in [J. Lond. Math. Soc. 113 (2026)], and a lower
bound of \(2m\) for arbitrary \(m\) was established by
Chen et al. [arXiv:2608.19542]. In this paper, we construct
systems with at least \(4m-2\) hyperbolic crossing limit
cycles for every positive integer \(m\), thereby disproving
Tonnelier's conjecture for \(m\geq2\). 
We also establish the uniform upper bound \(2^{224(m+1)^2}\) for the total number of crossing, grazing, sliding, and composite limit cycles.
	\end{abstract}
    
	\maketitle
	
	\section{Introduction and main results}
	Counting limit cycles is a central problem in the qualitative theory of planar differential equations. In the polynomial setting, the second part of Hilbert's sixteenth problem asks how the number and configuration of limit cycles depend on the degree of the vector field. As emphasized in the survey of Ilyashenko~\cite{I02}, finiteness for each individual system, a uniform bound over a prescribed class, and an explicit estimate for that bound are distinct questions. For piecewise linear systems, the number of switching values provides a natural measure of complexity. This leads to an analogous counting problem: to what extent does this number control the coexistence of limit cycles?
    
    Piecewise linear models combine explicitly solvable regional dynamics with nonlinear oscillatory behavior. Their role in the analysis of oscillations, including the construction of periodic orbits, is discussed in the recent survey of Coombes et al.~\cite{Coombes24}. Explicit regional flows do not, however, make the global matching problem linear. A periodic orbit must satisfy matching conditions at successive switching lines, and its itinerary can change when it becomes tangent to a line or acquires a sliding segment. These interactions belong to the broader theory of nonsmooth bifurcations reviewed by di Bernardo et al.~\cite{dB08}.
    
    In this paper, we study the discontinuous piecewise linear Li\'enard system
		\begin{equation}\label{eq:system}
			\dot x=y-F(x),\qquad \dot y=-x,
		\end{equation}
    where $F$ is piecewise linear and has exactly $m$ jump points and no fold points, in the terminology recalled below. We seek bounds depending only on $m$, uniformly over all admissible affine pieces, jump sizes, and switching locations.

   For system~\eqref{eq:system}, the problem of bounding the number of limit cycles in terms of the number of switching values goes back to Tonnelier~\cite{T02}. Motivated by McKean's piecewise linear version of the FitzHugh--Nagumo model, he conjectured that there are at most $n$ limit cycles when $F$ has $n$ fold points and no jump points, and at most $2m$ when it has $m$ jump points and no fold points. In the continuous case, the one-fold-point bound follows from Llibre et al.~\cite{LPV15}, while Chen et al.~\cite{CFJM26} proved that the maximum is exactly two for two fold points. The development of this theory includes rigorous constructions of two cycles in asymmetric three-zone systems~\cite{LPV15,LPV19} and exact cycle counts and global phase portraits in parameter regions of related three-zone Li\'enard families~\cite{CJT24}. 
    In the companion paper~\cite{CCJL26}, Cai et al. established the lower bound \(\lvert3n-4\rvert\) for the maximum number of limit cycles in the continuous case, disproving the continuous part of Tonnelier's conjecture for \(n\geq3\). They also obtained an explicit uniform finite upper bound depending only on \(n\).

    For the discontinuous case, Chen et al.~\cite{CDFZ26} proved that the maximum is exactly two for a single jump point; their analysis also gives global phase portraits and bifurcations, including the coexistence of crossing and sliding limit cycles. Chen et al.~\cite{CLZZ26} obtained the lower bound $2m$ as part of a study allowing both jump and fold points. 

Existence and uniqueness theory for nonsmooth Li\'enard equations provides several tools for studying their periodic orbits. Llibre et al.~\cite{LPT08} established nonexistence, existence, and uniqueness criteria for Li\'enard equations allowing discontinuities under structural hypotheses. Freire et al.~\cite{FPT12} developed canonical forms for discontinuous planar piecewise linear systems with two zones and analyzed their periodic dynamics. Li et al.~\cite{LCC20} studied crossing periodic orbits of nonsmooth Li\'enard systems with one or several equilibria, with applications to piecewise linear systems without sliding sets.

    Uniform upper bounds form a second line of research. Carmona et al.~\cite{CFN23} proved that a general planar piecewise linear system with two zones separated by a straight line has at most eight crossing limit cycles. Their proof combines an integral characterization of Poincar\'e half-maps with an extension of Khovanskii's theory. In a related setting, Bravo and Trinidad-Forte~\cite{BTF26} used Pfaffian estimates to bound crossing limit cycles of piecewise linear scalar equations on the cylinder with trigonometric polynomial coefficients. 
    The present paper addresses two questions: whether more than $2m$ limit cycles can coexist with exactly $m$ jump points and no fold points, and whether the total number of limit cycles, including crossing, grazing, sliding, and composite cycles, admits an explicit uniform bound depending only on $m$. Already for $m=2$, our construction yields six hyperbolic crossing limit cycles, exceeding the conjectured bound of four. The construction gives counterexamples for every $m\geq2$.

	We first recall the terminology used for the switching values of \(F\).
		\begin{definition}
			Let \(F:\mathbb{R}\to\mathbb{R}\) be a piecewise linear function, and let \(x_i\) be a switching value of \(F\), corresponding to the switching line \(x=x_i\) of system~\eqref{eq:system}. Denote 
			\[ 
			F(x_i^-)=\lim_{x\to x_i^-}F(x), \qquad F(x_i^+)=\lim_{x\to x_i^+}F(x).
			\] 
			When \(F\) is continuous at \(x_i\), denote its one-sided
			derivatives by 
			\[
			F'(x_i^-) = \lim_{h\to0^-} \frac{F(x_i+h)-F(x_i)}{h}, \qquad F'(x_i^+) = \lim_{h\to0^+} \frac{F(x_i+h)-F(x_i)}{h}. 
			\] 
			\begin{enumerate}[
				labelindent=0pt,
				labelwidth=1.5em,
				labelsep=0.4em,
				leftmargin=1.9em,
				align=left
				]
				\item[\textnormal{\textbf{(i)}}] The switching value \(x_i\) is called a \emph{fold point}
				of \(F\) if
				\(F(x_i^-)=F(x_i^+)=F(x_i)\) and \(F'(x_i^-)\neq F'(x_i^+).\)
				\item[\textnormal{\textbf{(ii)}}] The switching value \(x_i\) is called a \emph{jump point}
				of \(F\) if
				\(
				F(x_i^-)\neq F(x_i^+).
				\)	
			\end{enumerate}
		\end{definition}
	Let \(\Pm\) be the class of piecewise linear functions with exactly \(m\) jump points and no fold points. 
    We interpret system~\eqref{eq:system} in the sense of Filippov
\cite{Filippov}. On a switching line \(L_i=\{x=x_i\}\), the
corresponding Filippov set is
\begin{equation}\label{eq:filippov}
\mathcal F(x_i,y)
=
\left\{
\lambda\bigl(y-F(x_i^-),-x_i\bigr)
+(1-\lambda)\bigl(y-F(x_i^+),-x_i\bigr):
0\leq\lambda\leq1
\right\}.
\end{equation}

For a jump point \(x_i\), the crossing and
sliding sets on \(L_i\) are defined by
\[
L_i^c
=
\left\{
(x,y)\in L_i :
y<\min\{F(x_i^-),F(x_i^+)\}
\ \text{or}\
y>\max\{F(x_i^-),F(x_i^+)\}
\right\}
\]
and
\[
L_i^s
=
\left\{
(x,y)\in L_i :
\min\{F(x_i^-),F(x_i^+)\}
<
y
<
\max\{F(x_i^-),F(x_i^+)\}
\right\},
\]
respectively. Throughout the paper, a periodic orbit is understood as a simple closed Filippov orbit. For a periodic orbit $\Gamma$, let $\Gamma_s$
denote the union of the relative interiors
of its nonconstant sliding segments, and set
\(
\Gamma_{\mathrm{ns}}=\Gamma\setminus\Gamma_s.
\)
We further classify limit cycles into crossing limit cycles, grazing
limit cycles, sliding limit cycles, and composite limit cycles.
On \(x=x_i\), the sliding motion satisfies \(\dot y=-x_i\): it is downward for \(x_i>0\), upward for \(x_i<0\), and stationary for \(x_i=0\).

\begin{enumerate}
	\item[(a)]
	A periodic orbit \(\Gamma\) is called a \emph{crossing limit cycle} if it crosses the switching line transversally and 
    there exists a neighborhood of \(\Gamma\) containing no other periodic orbits
	
	\item[(b)]
	A periodic orbit $\Gamma$ is a \emph{sliding
limit cycle} if it contains a nonconstant
sliding segment and $\Gamma_{\mathrm{ns}}$
has a neighborhood which does not contain
the entire nonsliding part of any other
sliding periodic orbit.

	\item[(c)]
	A periodic orbit $\Gamma$ is a \emph{grazing limit cycle}
if it is tangent to a switching line at some point
and $\Gamma_{\mathrm{ns}}$ has an open neighborhood
which does not contain the entire nonsliding part
of any other periodic orbit.

	\item[(d)]
	A periodic orbit $\Gamma$ is a \emph{composite limit
cycle} if at least two of the three behaviors,
namely crossing, grazing, and sliding, occur
simultaneously on it, and $\Gamma_{\mathrm{ns}}$
has an open neighborhood which does not contain
the entire nonsliding part of any other
periodic orbit.
\end{enumerate}
These classifications need not be mutually
exclusive, and each distinct geometric
periodic orbit is counted only once.
	For $F\in\Pm$, let $H_m$ be the maximum number of limit cycles to system~\eqref{eq:system} according to \textup{(a)--(d)}.
Our main results give a constructive lower bound and an explicit uniform upper bound for \(H_m\).

			\begin{theorem}\label{thm:main}
				For every positive integer $m$, there exists $F\in\Pm$ such that system~\eqref{eq:system} has at least $4m-2$ hyperbolic crossing limit cycles. Moreover, all jump points of $F$ can be chosen to be positive. Consequently,
				\[
				H_m\ge 4m-2.
				\]
			\end{theorem}
The second result gives a uniform finite upper bound for arbitrary $F\in\Pm$.			
			\begin{theorem}\label{thm:upper}
				For every positive integer $m$, 
				\[
				H_m\leq B(m)< 2^{224(m+1)^2},
				\]
				where 
			\[
            B(m)= 8[10(m+1)]^{4m+6} 2^{32(m+1)^2-4(m+1)} 10^{64(m+1)-1} [768(m+1)]^{8(m+1)}.
            \]
			\end{theorem}

           For $m\geq2$, the inequality $4m-2>2m$ disproves the discontinuous part of Tonnelier's conjecture. The upper bound establishes uniform finiteness for arbitrary $m$, including periodic orbits with tangencies and sliding segments.

           The lower-bound proof perturbs the linear center and reduces the construction to simple zeros of a first-order displacement function. Averaging theory provides the passage from simple zeros to limit cycles~\cite{LMN15}. The main difficulty is to create new simple zeros without destroying those already obtained. Moment cancellation suppresses the leading terms of the displacement at large amplitudes, while separation of the amplitude scales preserves the earlier zeros. An induction creates three new positive simple zeros at each additional jump, producing at least $3m-3$ zeros. A final perturbation of the affine pieces creates $m+1$ additional simple zeros at larger amplitudes without moving the switching values.
           
           For the upper bound, the identity $\dot y=-x$ restricts the geometry of periodic orbits and bounds the number of arcs in their canonical decomposition. Linear-flow arcs admit Pfaffian representations, as in the continuous companion paper~\cite{CCJL26}. In the discontinuous setting, endpoint matching must be supplemented by conditions ensuring that the linear-flow and sliding arcs form a valid Filippov orbit. For each combinatorial type, these conditions define a semi-Pfaffian set. The isolation of a limit cycle allows us to bound the number of distinct geometric cycles by the number of connected components of this set, even when a cycle admits several encodings. The geometric arc bound makes the complexity of the Pfaffian description depend only on $m$, and connected-component estimates~\cite{Khovanskii,GV,GVZ04} give the uniform upper bound.

The paper is organized as follows. Section~2 develops the multiscale construction and proves Theorem~\ref{thm:main}. Section~3 establishes the geometric decomposition and Pfaffian encoding and proves Theorem~\ref{thm:upper}.

	\section{The lower bound}
	For the lower-bound construction, consider the perturbation of the
	linear center
		\begin{equation}\label{eq:perturbed}
		\dot x=y-\varepsilon F(x),\qquad \dot y=-x.
	\end{equation}
	
	Let $h:(0,\infty)\to\R$ be piecewise linear, with $h(x)=2x$ in a right neighborhood of zero. Define
	\begin{equation}\label{eq:f-from-h}
		F(x)=\begin{cases}
			x,&x\leq0,\\
			h(x)-x,&x>0.
		\end{cases}
	\end{equation}
	Then $F(x)=x$ in a neighborhood of the origin, and $F(x)-F(-x)=h(x)$ for $x>0$. 
    
    We first construct the return map on the section
\(\Sigma=\{(0,r):r>0\}\). Let \(J\subset(0,\infty)\)
be a compact interval with nonempty interior containing
no switching value of \(h\). For \(\varepsilon=0\),
the circles \(x^2+y^2=r^2\), \(r\in J\), cross the
same switching lines in the same order, and all these
intersections are transverse.

For sufficiently small \(|\varepsilon|\), the two
one-sided normal components at the corresponding
crossings remain nonzero and have the same sign.
Thus nearby trajectories cross the switching lines
uniquely and transversely. On each linear region,
the flow depends analytically on the initial condition
and on \(\varepsilon\). By transversality and the
implicit function theorem, the successive hitting
times have the same analytic dependence.

For the trajectory starting from \((0,r)\), write its
first return to \(\Sigma\) as \((0,P_\varepsilon(r))\)
and define the displacement function by
\[
\Delta_\varepsilon(r)=P_\varepsilon(r)-r.
\]
The composition of the successive transition maps shows
that \(P_\varepsilon(r)\) is jointly real analytic
on a neighborhood of \(J\times\{0\}\), with \(P_0(r)=r\).
The following lemma gives its first-order expansion.
For averaging theory in discontinuous piecewise differential systems, see~\cite{LMN15}.
\begin{lemma}\label{lem:averaging}
For system~\eqref{eq:perturbed}, with \(F\) defined by
\eqref{eq:f-from-h}, the displacement function satisfies
\begin{equation}\label{eq:return-expansion}
\Delta_\varepsilon(r)
=-\varepsilon rD_h(r)+\mathcal O(\varepsilon^2),
\end{equation}
as \(\varepsilon\to0\), uniformly for \(r\) in every
compact interval \(J\subset(0,\infty)\) containing
no switching value of \(h\), where
\begin{equation}\label{eq:D}
D_h(r)=\frac{2}{r^2}
\int_0^r\frac{xh(x)}{\sqrt{r^2-x^2}}\dd x,
\qquad r>0.
\end{equation}
The function \(D_h\) is real analytic away from the
switching values of \(h\).
\end{lemma}

\begin{proof}
For the energy function \(E(x,y)=(x^2+y^2)/2\),
system~\eqref{eq:perturbed} gives
\begin{equation}\label{eq:energy-identity}
\dot E=x\dot x+y\dot y=-\varepsilon xF(x).
\end{equation}
For \(\varepsilon=0\), the solution through \((0,r)\)
is \(x(t)=r\sin t\), \(y(t)=r\cos t\). Since
\(F(x)-F(-x)=h(x)\) for \(x>0\), we have
\begin{align*}
\int_0^{2\pi}x(t)F(x(t))\dd t
=2\int_{-r}^{r}
\frac{xF(x)}{\sqrt{r^2-x^2}}\dd x=2\int_0^r
\frac{xh(x)}{\sqrt{r^2-x^2}}\dd x
=r^2D_h(r).
\end{align*}
Integrating \eqref{eq:energy-identity} over one return
and using the analytic dependence of the flow segments
and hitting times gives
\[
\frac{P_\varepsilon(r)^2-r^2}{2}
=-\varepsilon r^2D_h(r)
+\mathcal O(\varepsilon^2).
\]
Since \(J\) is bounded away from zero and
\(P_\varepsilon(r)+r
=2r+\mathcal O(\varepsilon)\), division by
\((P_\varepsilon(r)+r)/2\) yields
\eqref{eq:return-expansion}.

By the joint analyticity of \(P_\varepsilon(r)\),
\eqref{eq:return-expansion} also holds after differentiation
with respect to \(r\), with a uniform
\(\mathcal O(\varepsilon^2)\) remainder on \(J\).

Finally, joint analyticity of \(P_\varepsilon(r)\)
and \eqref{eq:return-expansion} imply that
\(D_h(r)=-r^{-1}\partial_\varepsilon
P_\varepsilon(r)|_{\varepsilon=0}\) is real analytic
on the interior of \(J\). Since \(J\) is arbitrary,
the last assertion follows.
\end{proof}

To prove Theorem~\ref{thm:main}, we construct \(h\) so that
the corresponding \(F\), defined by \eqref{eq:f-from-h},
has exactly \(m\) jump points and \(D_h\) has at least
\(4m-2\) positive simple zeros away from these points. For sufficiently small \(\varepsilon>0\), these zeros
correspond to distinct hyperbolic crossing limit cycles
of system~\eqref{eq:perturbed}.

The construction has two steps. 
First, we construct a function \(h\) that vanishes for all sufficiently large \(x\) and has its first \(m-1\) moments vanishing, 
adding one jump point and three positive simple zeros at each inductive
step to obtain at least \(3m-3\) such zeros.
Second, we perturb the affine pieces of \(h\) while keeping
its switching values fixed to create \(m+1\) additional
positive simple zeros of \(D_h\) at a larger scale,
preserving the existing ones.
	
If \(h=0\) on \((L,\infty)\), the binomial expansion yields,
for \(r>L\),
\begin{equation}\label{eq:moment-expansion}
D_h(r)=2\sum_{j=0}^\infty
\frac{\gamma_j I_j(h)}{r^{2j+3}}.
\end{equation}
where
\[
\gamma_j=\frac{\binom{2j}{j}}{4^j},
\qquad
I_j(h)=\int_0^L x^{2j+1}h(x)\dd x,
\qquad j=0,1,2,\cdots.
\]
The series and its first derivative converge uniformly on compact subsets of \((L,\infty)\).

Thus the moments \(I_j(h)\) determine the leading
terms of \(D_h\) at infinity. To cancel or prescribe
these terms, we add constants to selected affine
pieces of \(h\). The changes in the moments depend
linearly on these constants. The following lemma
shows that the corresponding coefficient matrix
is invertible.

\begin{lemma}\label{lem:moment-matrix}
Let $0<b_0<b_1<\cdots<b_n$ and define
\[
\mathcal M_{ji}=\int_{b_{i-1}}^{b_i}x^{2j+1}\dd x,
\qquad 0\leq j\leq n-1,\quad 1\leq i\leq n.
\]
Then $\det\mathcal M>0$. Consequently, for any prescribed values of
\(I_0,\cdots,I_{n-1}\), there is a unique choice of constants
added on these \(n\) intervals that realizes them.
\end{lemma}

\begin{proof}
By multilinearity of the determinant and the Vandermonde identity,
\[
\det\mathcal M
=
\int_{b_0}^{b_1}\cdots\int_{b_{n-1}}^{b_n}
\left(\prod_{i=1}^n x_i\right)
\left(\prod_{1\le p<q\le n}(x_q^2-x_p^2)\right)
\dd x_n\cdots\dd x_1>0,
\]
since \(0<x_1<\cdots<x_n\) on the integration region.
The last assertion follows from the invertibility of \(\mathcal M\).
\end{proof}
To increase the number of simple zeros, we extend
the support of \(h\) by adding an affine piece ending
at a large radius \(R\), and adjust the existing
pieces to cancel the first \(k\) moments.
After the rescaling \(r=Rt\), the resulting averaged
function will converge, up to a nonzero factor,
to \(\Phi_k\) defined below.
The next lemma identifies a choice of its parameter
for which this limit has three simple zeros.

\begin{lemma}\label{lem:profile}
For each integer \(k\geq1\), there exists
\(q\in((2k+2)/(2k+3),1)\) such that the continuous function
\begin{equation}\label{eq:profile}
\Phi_k(t)=\begin{cases}
-A_k(t)+qB_k(t),&0<t\leq1,\\[1mm]
\displaystyle 2\sum_{j=k}^\infty\gamma_j
\left(\frac{q}{2j+2}-\frac{1}{2j+3}\right)t^{-2j-3},&t>1,
\end{cases}
\end{equation}
where
\[
A_k(t)=\frac\pi2-2\sum_{j=0}^{k-1}\frac{\gamma_j}{(2j+3)t^{2j+3}},\quad
B_k(t)=\frac2t-2\sum_{j=0}^{k-1}\frac{\gamma_j}{(2j+2)t^{2j+3}},
\]
has exactly two simple zeros in $(0,1)$ and exactly one simple zero in $(1,\infty)$. 
\end{lemma}

\begin{proof}
Recall that
\(\gamma_j=\binom{2j}{j}/4^j\).
Termwise integration of
\((1-x^2)^{-1/2}=\sum_{j\geq0}\gamma_jx^{2j}\)
after multiplication by \(x^2\) and \(x\) gives
\begin{equation}\label{eq:series-identities}
\sum_{j=0}^\infty\frac{\gamma_j}{2j+3}=\frac\pi4,
\qquad
\sum_{j=0}^\infty\frac{\gamma_j}{2j+2}=1.
\end{equation}
Thus \(A_k(1),B_k(1)>0\), and
\[
q_k^*:=\frac{A_k(1)}{B_k(1)}
=\frac{\displaystyle\sum_{j=k}^\infty
\frac{\gamma_j}{2j+2}\frac{2j+2}{2j+3}}
{\displaystyle\sum_{j=k}^\infty\frac{\gamma_j}{2j+2}}
=\sum_{j=k}^\infty w_j\frac{2j+2}{2j+3},
\]
where
\[
w_j=
\frac{\gamma_j/(2j+2)}
{\displaystyle\sum_{\ell=k}^\infty\gamma_\ell/(2\ell+2)},
\qquad j\ge k.
\]
Since \(w_j>0\), \(\sum_{j=k}^\infty w_j=1\), and
\((2j+2)/(2j+3)\) increases strictly to \(1\), we obtain
\((2k+2)/(2k+3)<q_k^*<1\).

For $q=q_k^*$, one has $\Phi_k(1)=0$ and \(\Phi_k'(1^-)=2(q-1)\sum_{j=0}^{k-1}\gamma_j-qB_k(1)<0\).
Hence $\Phi_k(t_*)>0$ for some $t_*\in(0,1)$. Decrease $q$ slightly, so that $q>(2k+2)/(2k+3)$, $\Phi_k(t_*)>0$, and $\Phi_k(1)<0$. Since $q>2k/(2k+1)$, the leading singular term gives $\Phi_k(t)\to-\infty$ as $t\to0^+$. Consequently, \(\Phi_k\) has at least two zeros in \((0,1)\).

For \(0<t<1\), set
\[
P_k(t):=t^{2k+1}\Phi_k(t)
=-\frac\pi2t^{2k+1}+2qt^{2k}
+2\sum_{j=0}^{k-1}\gamma_j
\left(\frac{1}{2j+3}-\frac{q}{2j+2}\right)t^{2k-2j-2}.
\]
Since \(q>(2k+2)/(2k+3)>(2j+2)/(2j+3)\) for
\(0\leq j\leq k-1\), all coefficients in the sum are negative.
Thus the nonzero coefficients of \(P_k\), listed in decreasing
powers of \(t\), have signs \(-,+,-,\cdots,-\), with exactly two
sign changes. By Descartes' rule of signs, the polynomial has at most
two positive zeros, counted with multiplicity. 

Hence the two zeros of \(\Phi_k\)
already found in \((0,1)\) are the only zeros there, and both are simple.

Identities~\eqref{eq:series-identities} show that the two expressions in \eqref{eq:profile} agree at $t=1$. As \(t\to\infty\), 
\[
\Phi_k(t)=2\gamma_k\left(\frac{q}{2k+2}-\frac{1}{2k+3}\right)t^{-2k-3} +\mathcal O(t^{-2k-5}),
\]
whose leading coefficient is positive. Together with \(\Phi_k(1)<0\), this shows that
\(\Phi_k\) has at least one zero in \((1,\infty)\).

To prove that \(\Phi_k\) has exactly one zero in \((1,\infty)\)
and that this zero is simple, set
\[
u_j=\gamma_j\left(\frac{q}{2j+2}-\frac{1}{2j+3}\right),
\qquad S(z)=\sum_{j=k}^\infty u_jz^j,\quad 0<z<1.
\]
The sequence $(2j+2)/(2j+3)$ increases strictly to $1$. The nonzero coefficients \(u_j\), listed in increasing order of \(j\),
change sign exactly once, from positive to negative. 

Let $N$ be the first index with $u_N<0$. Then
\[
\frac{\mathrm d}{\mathrm dz}\bigl(z^{-N}S(z)\bigr)
=\sum_{j=k}^\infty (j-N)u_jz^{j-N-1}<0,
\]
since every term in the sum is nonpositive, and the term with \(j=k\)
is strictly negative. Thus \(z^{-N}S(z)\) is strictly decreasing on \((0,1)\),
so \(S\) has at most one zero there.
Moreover, at any zero \(z_0\) of \(S\),
\[
z_0^{-N}S'(z_0)
=\left.\frac{\mathrm d}{\mathrm dz}
\bigl(z^{-N}S(z)\bigr)\right|_{z=z_0}<0,
\]
which shows that this zero is simple.
Since \(\Phi_k(t)=2t^{-3}S(t^{-2})\) for \(t>1\),
and the change of variables \(z=t^{-2}\) preserves
zeros and their multiplicities, \(\Phi_k\) has at most
one zero in \((1,\infty)\), and any such zero is simple.

Together with the existence established above, this proves
that \(\Phi_k\) has exactly one simple zero in \((1,\infty)\). This completes the proof.
\end{proof}

We now combine Lemmas~\ref{lem:moment-matrix} and~\ref{lem:profile}
to carry out the first step of the construction.
At each stage, we impose the required moment conditions and
create three new simple zeros while preserving the existing ones.

\begin{proposition}\label{prop:induction}
For each integer \(k\geq1\), there exist
\(0<a_1<\cdots<a_k=L\) and a piecewise linear function \(h\)
such that:
\begin{enumerate}[
labelindent=0pt,
labelwidth=1.5em,
labelsep=0.4em,
leftmargin=1.9em,
align=left
]
\item[\textnormal{\textbf{(i)}}] \(h(x)=2x\) for \(0<x<a_1\), \(h(x)=0\) for \(x>L\),
and its only switching values are the \(k\) nonzero jump points
\(a_1,\cdots,a_k\);
\item[\textnormal{\textbf{(ii)}}] \(I_0(h)=\cdots=I_{k-2}(h)=0\) and \(I_{k-1}(h)\neq0\);
\item[\textnormal{\textbf{(iii)}}] \(D_h\) has at least \(3k-3\) positive simple zeros,
all distinct from \(a_1,\cdots,a_k\).
\end{enumerate}
The vanishing moment condition in \textnormal{\textbf{(ii)}} is void for \(k=1\).
\end{proposition}

\begin{proof}
We argue by induction on \(k\). For $k=1$, take $a_1=1$ and
\[
h(x)=2x\quad(0<x<1),\qquad h(x)=0\quad(x>1).
\]
Then $I_0(h)=2/3\neq0$, and no zero is required in \textbf{(iii)}.

Assume that the assertion holds for some integer \(k\geq1\).
We prove it for \(k+1\).
Let \(h\) satisfy \textbf{(i)}--\textbf{(iii)} for this \(k\), and set \(J=I_{k-1}(h)\neq0\).

Fix \(q\in((2k+2)/(2k+3),1)\) such that \(\Phi_k\)
has the three simple zeros given by Lemma~\ref{lem:profile}.
For a large parameter \(R>L\), set
\[
\ell_R(x)=(\sgn J)b_R
\bigl(R^{-2k-1}x-qR^{-2k}\bigr),
\]
and define
\[
h_R(x)=
\begin{cases}
2x, & 0<x<a_1,\\
h(x)+d_i, & a_{i-1}<x<a_i,\quad i=2,\cdots,k,\\
\ell_R(x), & L<x<R,\\
0, & x>R,
\end{cases}
\]
where \(b_R>0\) and \(d_2,\cdots,d_k\) are to be determined. For \(0\leq j\leq k-2\), the added interval contributes
\[
\begin{aligned}
\int_L^R x^{2j+1}\ell_R(x)\dd x
=(\sgn J)b_RR^{2j-2k+2}
\left(\frac{1}{2j+3}-\frac{q}{2j+2}\right)
+\mathcal{O}(b_RR^{-2k})=\mathcal{O}(b_RR^{-2}).
\end{aligned}
\]
Lemma~\ref{lem:moment-matrix} uniquely determines $d_2,\cdots,d_k$ so that the first $k-1$ moments of $h_R$ vanish, with \(d_i=\mathcal{O}(b_RR^{-2})\), \(i=2,\cdots,k\). Moreover, the constants \(d_2,\cdots,d_k\) depend linearly on \(b_R\).
When \(k=1\), there are no constants \(d_i\) to determine,
and this step is omitted.

To complete the verification of \textbf{(ii)} for \(k+1\),
we determine \(b_R\) by requiring \(I_{k-1}(h_R)=0\).
This condition takes the form
\[
0=J+(\sgn J)b_R\left(\frac{1}{2k+1}-\frac{q}{2k}\right)
+\mathcal{O}(b_RR^{-2}).
\]
This is an exact linear equation in $b_R$, with the displayed asymptotic coefficient. Since $q>2k/(2k+1)$, its unique solution for large $R$ is positive and satisfies
\begin{equation}\label{eq:bR}
b_R=b_0+\mathcal{O}(R^{-2}),\qquad
b_0:=\frac{|J|}{q/(2k)-1/(2k+1)}>0.
\end{equation}
Hence $I_0(h_R)=\cdots=I_{k-1}(h_R)=0$. The next moment is
\[
I_k(h_R)=(\sgn J)b_0R^2
\left(\frac{1}{2k+3}-\frac{q}{2k+2}\right)
+\mathcal{O}(1)\neq0
\]
for sufficiently large \(R\), since \(q>(2k+2)/(2k+3)\).
Thus \(h_R\) satisfies \textbf{(ii)} with \(k\) replaced by \(k+1\).

We next verify \textbf{(i)}.
Since \(d_i\to0\) and \(\ell_R(L)\to0\) as \(R\to\infty\),
all old jumps remain nonzero for sufficiently large \(R\).
At the new endpoint \(R\), we have
\[
h_R(R^-)=(\sgn J)b_R(1-q)R^{-2k}\neq0,
\qquad h_R(R^+)=0.
\]
By construction, \(h_R(x)=2x\) for \(0<x<a_1\),
\(h_R(x)=0\) for \(x>R\), and there are no other
switching values. Thus \(h_R\) satisfies \textbf{(i)}
with \(a_{k+1}=R\).

Finally, we verify \textbf{(iii)}.
We first show that the \(3k-3\) positive simple zeros of \(D_h\)
persist as simple zeros of \(D_{h_R}\). 

Since \(d_2,\cdots,d_k\) and both coefficients of \(\ell_R\)
tend to zero, evaluating \eqref{eq:D} separately on each
affine piece gives
\[
D_{h_R}\longrightarrow D_h,
\qquad
D_{h_R}'\longrightarrow D_h'
\qquad\text{as }R\to\infty,
\]
uniformly on every fixed compact interval
\(K\subset(0,\infty)\setminus\{a_1,\cdots,a_k\}\).

Choose small disjoint compact neighborhoods of $3k-3$ old simple zeros on which $D_h'$ has constant nonzero sign and the endpoint values of $D_h$ have opposite signs. These properties persist for large $R$, so each neighborhood contains one simple zero of $D_{h_R}$.

We next show that \(D_{h_R}\) has three additional
positive simple zeros at scale \(R\). Set
\[
u_R=(\sgn J)b_RR^{-2k-1},
\qquad v_R=-(\sgn J)qb_RR^{-2k}.
\]
Thus \(\ell_R(x)=u_Rx+v_R\).
On the interval \((0,L)\), define \(g_R=h_R-\ell_R\),
and extend \(g_R\) by zero outside \([0,L]\). Let
\[
K_j(R)=\int_0^L x^{2j+1}g_R(x)\dd x.
\]
Since \(I_j(h_R)=0\) for \(0\leq j\leq k-1\),
we obtain
\begin{equation}\label{eq:residual-moments}
K_j(R)=-\frac{u_RR^{2j+3}}{2j+3}
-\frac{v_RR^{2j+2}}{2j+2},\qquad 0\leq j\leq k-1.
\end{equation}
The functions \(g_R\) vanish outside \([0,L]\) and are
uniformly bounded for all sufficiently large \(R\).

For \(0<t<1\) with \(Rt>L\), evaluating \eqref{eq:D}
for \(h_R\) over \((0,Rt)\) and expanding the contribution of \(g_R\) gives
\[
D_{h_R}(Rt)=\frac{\pi u_R}{2}+\frac{2v_R}{Rt}
+2\sum_{j=0}^\infty\frac{\gamma_jK_j(R)}{(Rt)^{2j+3}}.
\]
Since the functions \(g_R\) are uniformly bounded on \([0,L]\),
we have \(|K_j(R)|\leq C L^{2j+2}\), where \(C\) is a constant independent of \(j\) and \(R\).
Geometric-series bounds then give
\(2\sum_{j=k}^\infty \gamma_jK_j(R)/(Rt)^{2j+3}
=\mathcal O(R^{-2k-3})\),
uniformly for \(t\) in every compact set \(K\subset(0,1)\)
as \(R\to\infty\).
The same estimate holds for the derivative of this sum
with respect to \(t\).

Substituting \eqref{eq:residual-moments} for \(K_j(R)\),
\(j=0,\cdots,k-1\), in the above expression for \(D_{h_R}(Rt)\)
and multiplying by \(R^{2k+1}\), we obtain
\begin{equation}\label{eq:inner-scale}
R^{2k+1}D_{h_R}(Rt)
=(\sgn J)b_R\bigl(A_k(t)-qB_k(t)\bigr)
+\mathcal O(R^{-2}),
\end{equation}
where the remainder estimate holds uniformly for \(t\)
in every compact set \(K\subset(0,1)\) as \(R\to\infty\).

For \(t>1\), we apply \eqref{eq:moment-expansion} to \(h_R\)
at \(r=Rt>R\) and use \eqref{eq:residual-moments}
to eliminate the terms with \(j=0,\cdots,k-1\).
The remaining contribution of \(g_R\) and its first derivative
with respect to \(t\) satisfy the same
\(\mathcal O(R^{-2k-3})\) estimate as above,
uniformly on every compact set \(K\subset(1,\infty)\).
Multiplying by \(R^{2k+1}\), we obtain
\begin{equation}\label{eq:outer-scale}
\begin{aligned}
R^{2k+1}D_{h_R}(Rt)
=2(\sgn J)b_R\sum_{j=k}^\infty\gamma_j
\left(\frac{1}{2j+3}-\frac{q}{2j+2}\right)
t^{-2j-3}+\mathcal O(R^{-2}),
\end{aligned}
\end{equation}
where the remainder and its first derivative with respect
to \(t\) are uniformly \(\mathcal O(R^{-2})\) on \(K\).

Thus, in \eqref{eq:inner-scale} and \eqref{eq:outer-scale},
the remainders and their first derivatives with respect
to \(t\) are uniformly \(\mathcal O(R^{-2})\)
on the respective compact sets \(K\).

By \eqref{eq:bR}, \eqref{eq:inner-scale}, and \eqref{eq:outer-scale},
\[
R^{2k+1}D_{h_R}(Rt)
\longrightarrow-(\sgn J)b_0\Phi_k(t)
\qquad\text{as }R\to\infty.
\]
The convergence holds in
\(C_{\mathrm{loc}}^1\bigl((0,1)\cup(1,\infty)\bigr)\). Choose disjoint compact neighborhoods of the three simple zeros of $\Phi_k$, avoiding $t=1$. The same sign and derivative argument as above gives one new simple zero in each neighborhood. Their radii are $R(t_j+o(1))$. 

For large \(R\), they are distinct from all old zeros
and switching values, including the new interface \(R\).
Thus \(D_{h_R}\) has at least \(3k-3+3=3(k+1)-3\)
positive simple zeros away from all switching values.
This verifies \textbf{(iii)} for \(k+1\) and completes
the induction.
\end{proof}

\noindent{\bf Proof of Theorem~\ref{thm:main}.}
Apply Proposition~\ref{prop:induction} with \(k=m\).
The resulting function \(h\) has exactly \(m\) jump
points \(0<a_1<\cdots<a_m=L\), and \(D_h\) has at
least \(3m-3\) positive simple zeros away from these
points. We modify \(h\), keeping its switching values
fixed, to obtain \(m+1\) additional positive simple zeros.

Set \(J=I_{m-1}(h)\neq0\). Replace the rightmost zero
piece by \(ax+b\), and add constants to the old pieces
numbered \(2,\cdots,m\), keeping the first piece fixed.
Denote the resulting perturbation of \(h\), with the same
switching values, by \(\widetilde h\). Define the residual moments
\[
\mu_j=\int_0^L x^{2j+1}
\bigl(\widetilde h(x)-ax-b\bigr)\dd x.
\]
For fixed \(a,b\), Lemma~\ref{lem:moment-matrix}
permits independent prescription of
\(\mu_0,\cdots,\mu_{m-2}\). If \(m=1\), there are no
such moments or internal corrections.

Introduce a large scale \(T>0\) and set
\begin{equation}\label{eq:final-scaling}
a=\alpha T^{-2m-1},\qquad
b=\beta T^{-2m},\qquad
\mu_j=\nu_jT^{2j-2m+2},
\quad 0\leq j\leq m-2,
\end{equation}
where the constants \(\alpha,\beta,\nu_j\), independent of \(T\),
will be chosen below.

All prescribed moments tend to zero at rate
\(\mathcal O(T^{-2})\) or faster. Since the coefficient
matrix for the internal corrections is fixed and
invertible, these corrections are \(\mathcal O(T^{-2})\).
Consequently,
\begin{equation}\label{eq:mu-leading}
\mu_{m-1}=J+\mathcal O(T^{-2}).
\end{equation}

For \(r>L\), the contribution of \(ax+b\) on the
positive half-axis and that of the compactly supported
residual give
\[
D_{\widetilde h}(r)
=\frac{\pi a}{2}+\frac{2b}{r}
+2\sum_{j=0}^{\infty}
\frac{\gamma_j\mu_j}{r^{2j+3}}.
\]
At \(r=Tt\), equations
\eqref{eq:final-scaling}--\eqref{eq:mu-leading} imply
\begin{equation}\label{eq:G-convergence}
T^{2m+1}D_{\widetilde h}(Tt)
=G(t)+\mathcal O(T^{-2})
\end{equation}
on every compact set \(K\subset(0,\infty)\), where
\[
G(t)=\frac{\pi\alpha}{2}+\frac{2\beta}{t}
+\sum_{j=0}^{m-2}\frac{2\gamma_j\nu_j}{t^{2j+3}}
+\frac{2\gamma_{m-1}J}{t^{2m+1}}.
\]
The residual has fixed support and is uniformly bounded.
By the geometric-series estimate used in
Proposition~\ref{prop:induction}, the remainder in
\eqref{eq:G-convergence} and its first derivative
with respect to \(t\) are uniformly
\(\mathcal O(T^{-2})\) on \(K\).

Set
\[
Q(t):=t^{2m+1}G(t)=\frac{\pi\alpha}{2}t^{2m+1}
+2\beta t^{2m}
+\sum_{j=0}^{m-2}2\gamma_j\nu_jt^{2m-2j-2}
+2\gamma_{m-1}J.
\]
Then the powers of \(t\) appearing in \(Q\) are
\(0,2,4,\cdots,2m,2m+1\).
The constant term \(2\gamma_{m-1}J\) is fixed and nonzero,
whereas the remaining \(m+1\) coefficients are determined by
\(\alpha,\beta,\nu_0,\cdots,\nu_{m-2}\). We choose these parameters
so that
\begin{equation}\label{eq:interpolation}
Q(i)=0,\qquad i=1,\cdots,m+1.
\end{equation}
We claim that the interpolation matrix
\[
\mathcal A:=
\begin{pmatrix}
1^2&1^4&\cdots&1^{2m}&1^{2m+1}\\
2^2&2^4&\cdots&2^{2m}&2^{2m+1}\\
\vdots&\vdots&&\vdots&\vdots\\
(m+1)^2&(m+1)^4&\cdots&(m+1)^{2m}&(m+1)^{2m+1}
\end{pmatrix}
\]
is invertible. Suppose, to the contrary,
that it is singular. Then the associated homogeneous system has a nonzero
solution, and hence there exists a nonzero polynomial
\[
P(t)=c_1t^2+c_2t^4+\cdots+c_mt^{2m}+c_{m+1}t^{2m+1}
\]
such that
\[
P(i)=0,\qquad i=1,\cdots,m+1.
\]
Since \(P\) has at most \(m+1\) nonzero monomials, Descartes'
rule of signs implies that it has at most \(m\) positive zeros,
counted with multiplicity, a contradiction. Thus \eqref{eq:interpolation} uniquely determines
\(\alpha,\beta,\nu_0,\cdots,\nu_{m-2}\).

The resulting polynomial \(Q\) has at most \(m+2\)
nonzero monomials. Descartes' rule bounds its positive
zeros, counted with multiplicity, by \(m+1\).
The prescribed \(m+1\) distinct zeros are therefore
all simple, and \(G'(i)\neq0\), \(i=1,\cdots,m+1\).

For sufficiently large \(T\), all original jumps remain nonzero.
Moreover, on fixed compact neighborhoods of the old simple zeros
avoiding the switching values, \(D_{\widetilde h}\to D_h\) in \(C^1\).
Hence the \(3m-3\) old simple zeros persist.

By \eqref{eq:G-convergence}, \(D_{\widetilde h}\)
has simple zeros at radii
\[
r_i(T)=Ti+\mathcal O(T^{-1}),
\qquad i=1,\cdots,m+1.
\]
For sufficiently large \(T\), these radii exceed \(L\)
and all fixed neighborhoods of the old zeros, and
are pairwise distinct. Fix such a \(T\).
Then \(D_{\widetilde h}\) has at least
\[
(3m-3)+(m+1)=4m-2
\]
positive simple zeros, all distinct from the
switching values of \(\widetilde h\).

It remains to transfer these simple zeros of
\(D_{\widetilde h}\) to hyperbolic crossing limit cycles of the
perturbed system.

Define \(F\) from \(\widetilde h\) by
\eqref{eq:f-from-h}. Since the first piece of
\(\widetilde h\) is \(2x\), we have \(F(x)=x\)
for \(x<a_1\), so no switching occurs at the origin.
On every remaining interval, \(F=\widetilde h-x\)
is affine. Subtracting \(x\) preserves every jump;
hence \(F\) has exactly \(m\) jump points and no
fold points.

Let \(\rho_1,\cdots,\rho_{4m-2}\) be distinct positive
simple zeros of \(D_{\widetilde h}\) obtained above.
By Lemma~\ref{lem:averaging}, the quotient
\(\Delta_\varepsilon(r)/\varepsilon\) extends
analytically to \(\varepsilon=0\), with value
\(-rD_{\widetilde h}(r)\). At each \(\rho_i\), its
derivative with respect to \(r\) equals
\(-\rho_iD_{\widetilde h}'(\rho_i)\neq0\).
The implicit function theorem therefore gives a
fixed point
\[
r_i(\varepsilon)=\rho_i+\mathcal O(\varepsilon)
\]
of the return map for sufficiently small \(\varepsilon\).
Its multiplier satisfies
\[
P_\varepsilon'(r_i(\varepsilon))
=1-\varepsilon\rho_iD_{\widetilde h}'(\rho_i)
+\mathcal O(\varepsilon^2),
\]
which is positive and different from \(1\) for
sufficiently small \(\varepsilon>0\). Each fixed
point thus determines a hyperbolic crossing limit
cycle converging to \(x^2+y^2=\rho_i^2\).

Since the selected zeros are finite in number and
distinct, these limit cycles lie in pairwise disjoint
annular neighborhoods for one common range
\(0<\varepsilon<\varepsilon_0\).
Finally, multiplication by \(\varepsilon>0\)
preserves the jump points and introduces no fold
points, so \(\varepsilon F\in\Pm\). Therefore,
\[
H_m\geq4m-2.
\]
This proves the theorem.
\hfill\(\Box\)

A numerical phase portrait is presented to illustrate
Theorem~\ref{thm:main} in the case \(m=2\).
In this case, the theorem guarantees the existence of parameter values
for which system~\eqref{eq:system} admits at least six limit cycles.
For the numerical illustration, consider the system
\begin{equation}\label{sys:m2}
    \dot x=y-F(x),\qquad \dot y=-x,
\end{equation}
where
\[
F(x)=
\begin{cases}
    m_1x+d_1, & x\leq 1,\\
    m_2x+d_2, & 1<x\leq 5,\\
    m_3x+d_3, & x>5,
\end{cases}
\]
with
\[
\begin{aligned}
    m_1&=\varepsilon,
    &d_1&=0,\\
    m_2&=-\frac{99}{109}\varepsilon,
    &d_2&=\varepsilon\left(-\frac{81}{218}+\delta\right),\\
    m_3&=\varepsilon(a-1),
    &d_3&=\varepsilon b,
\end{aligned}
\]
and
\[
\varepsilon=10^{-3},\qquad
a=\frac{8}{109\pi\,50^5},\qquad
b=-\frac{239}{3270\,50^4},\qquad
\delta=\frac{115}{19620000}
       +\frac{125a}{36}+\frac{25b}{24}.
\]
These coefficients ensure that \(F\) has exactly two jump points,
at \(x=1\) and \(x=5\), and no fold points.
High-precision computations of the return map identify six
hyperbolic crossing limit cycles, whose stability alternates
from unstable for the innermost cycle to stable for the outermost one.

To display the six cycles together despite their differing amplitudes,
the nonlinear coordinate transformation
\(X=\operatorname{arcsinh}(x)\),
\(Y=\operatorname{arcsinh}(y)\) is applied.
The resulting phase portrait is shown in Figure~\ref{fig:m2},
where the square-like appearance of the cycles is a consequence
of this coordinate compression.
The tick labels indicate the original values of \(x\) and \(y\).

\begin{figure}[htbp]
    \centering
    \includegraphics[width=0.9\textwidth]
    {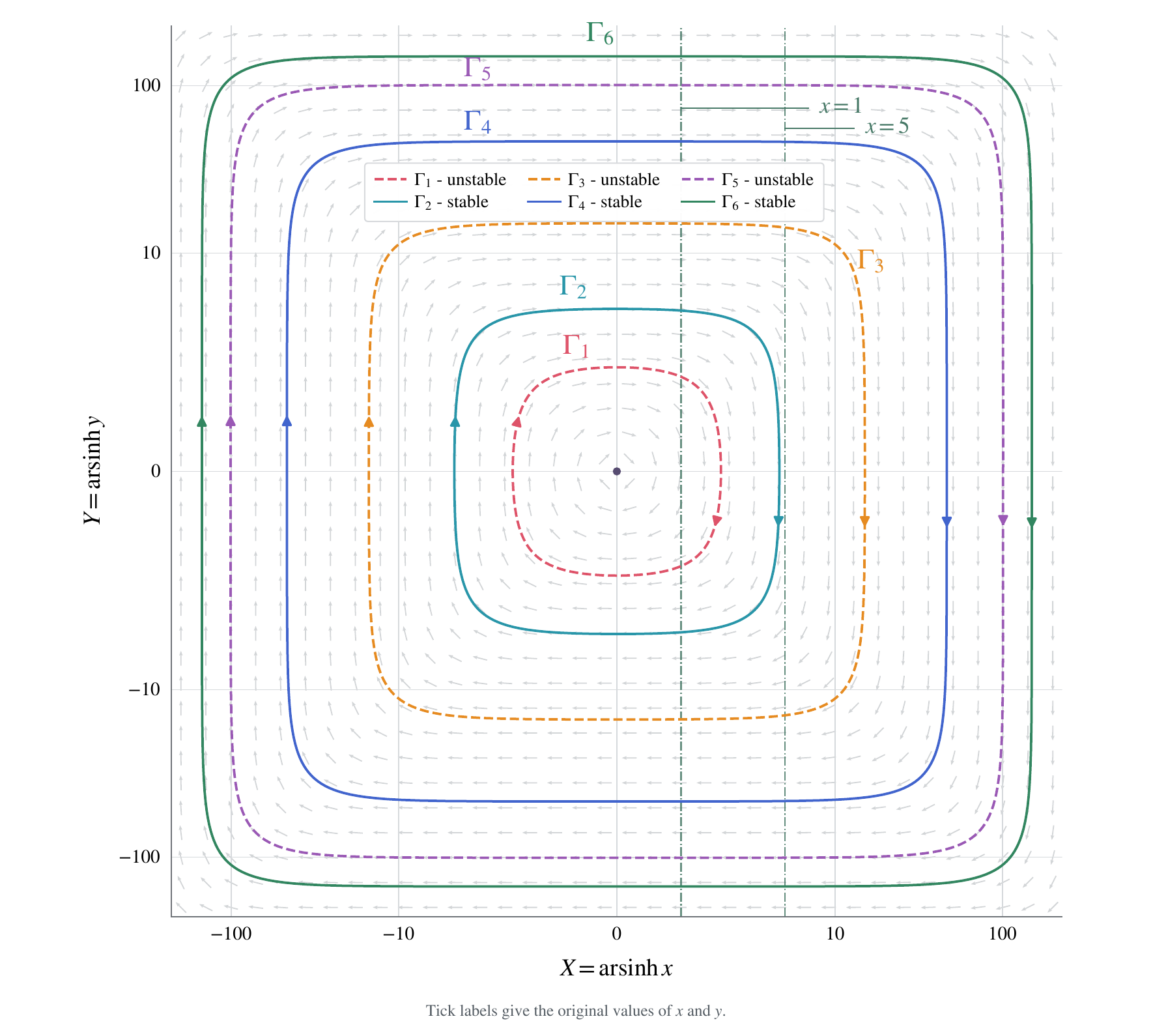}
    \caption{Numerical phase portrait of system~\eqref{sys:m2}.}
    \label{fig:m2}
\end{figure}

\section{The upper bound}
Throughout this section, periodic orbits are understood as simple closed Filippov orbits, namely, nonconstant periodic Filippov solutions whose
traces are simple closed curves.  Let \(F\in\Pm\) be arbitrary. Denote its jump points by
\(
\xi_1<\cdots<\xi_m,
\)
set \(\xi_0=-\infty\) and \(\xi_{m+1}=+\infty\), and write
\[
F(x)=a_i x+b_i,\qquad
\xi_i<x<\xi_{i+1},\qquad i=0,\cdots,m.
\]
We first bound the possible decompositions of periodic orbits, then
describe each decomposition by a global Pfaffian system, and finally
apply a connected-component estimate to obtain a bound depending only
on \(m\).

\begin{lemma}\label{lem:geometry}
Every simple periodic orbit of system~\eqref{eq:system} meets the
\(y\)-axis exactly twice. On its right half-plane arc, \(x\) is first
nondecreasing and then nonincreasing. On its left half-plane arc, \(x\)
is first nonincreasing and then nondecreasing.
\end{lemma}

\begin{proof}
If a solution has period \(T\), then
\[
0=y(T)-y(0)=-\int_0^T x(t)\dd t.
\]
Thus every nonconstant periodic solution enters both open half-planes;
otherwise \(x\) has one weak sign, which forces \(x\equiv0\) and then
\(y\) to be constant.

Set
\[
\beta_-=\min\{F(0^-),F(0^+)\},\qquad
\beta_+=\max\{F(0^-),F(0^+)\}.
\]
If zero is not a jump point, the two numbers coincide. A nonconstant
periodic orbit cannot meet
\[
S_0=\{(0,y):\beta_-\leq y\leq\beta_+\}.
\]
Indeed, the tangential Filippov vector in the interior of \(S_0\) is
zero. When \(F(0^-)<F(0^+)\), both adjacent vector fields point toward
the axis, so an orbit reaching \(S_0\) cannot leave it. When
\(F(0^-)>F(0^+)\), both fields point away from the axis, so an orbit
cannot enter \(S_0\). At an endpoint of \(S_0\), one adjacent linear
vector field has an equilibrium; uniqueness for that field, together
with the direction of the other field, again excludes a passage by a
nonconstant periodic orbit. The same conclusion follows from ordinary
uniqueness when \(\beta_-=\beta_+\).

Every other intersection with the \(y\)-axis is transverse. A crossing
from left to right has \(y>\beta_+\), whereas a crossing from right to
left has \(y<\beta_-\). These intersections are finite because they are
locally isolated and the orbit is compact. Their directions alternate
when the intersection points are ordered along the axis, by the Jordan
curve theorem. Since all left-to-right crossings lie above all
right-to-left crossings, there is exactly one crossing of each type.

It remains to study \(x\) on the two half-plane arcs. On the right arc,
\(y\) is strictly decreasing, including along every sliding segment.
If \(x\) had a valley, there would exist \(t_1<t_2<t_3\) such that
\[
x(t_2)<\min\{x(t_1),x(t_3)\}.
\]
Choose a regular value \(d\) strictly between these quantities and
different from all jump positions. The orbit would cross \(x=d\) first
to the left and later to the right, at times \(s_1<s_2\). Hence
\[
y(s_1)-F(d)=\dot x(s_1)\leq0,\qquad
y(s_2)-F(d)=\dot x(s_2)\geq0,
\]
contrary to \(y(s_2)<y(s_1)\). Therefore \(x\) is nondecreasing up to
its maximum and nonincreasing afterwards. The same argument on the left
arc, where \(y\) is strictly increasing, excludes a peak and proves the
last assertion.
\end{proof}
This geometric structure allows us to decompose every periodic orbit
into finitely many monotone linear-flow and sliding arcs.

For a periodic orbit, its \emph{canonical
decomposition} is obtained by cutting the
orbit at its two intersections with the
\(y\)-axis, at the endpoints of its
intersections with the switching lines,
and, if necessary, at the interior extrema
of \(x\) in the two open half-planes.
The resulting decomposition consists of
linear-flow arcs on which \(x\) is strictly
monotone and nonconstant sliding arcs.
Starting from the upper intersection
with the $y$-axis and following the orbit
in positive time, the cyclic sequence of
these arcs, together with their assigned
linear regions or switching lines and
their monotonicity directions, determines
the \emph{combinatorial type} $w$.

\begin{lemma}\label{lem:arc-count}
Every simple periodic orbit admits a canonical decomposition consisting
of at most \(2m+4\) linear-flow arcs on which \(x\) is strictly
monotone and at most \(2m\) nonconstant sliding arcs. Hence the total
number of arcs is at most \(4m+4\). Moreover, the number \(W(m)\) of possible combinatorial types satisfies
\[
    W(m)\leq [10(m+1)]^{4m+6}.
\]
\end{lemma}

\begin{proof}
By Lemma~\ref{lem:geometry}, \(x\) is nondecreasing and then
nonincreasing on the right half-plane arc, and nonincreasing and then
nondecreasing on the left one. Hence every nonzero switching line
\(x=\xi_i\) meets a simple periodic orbit in at most two connected
components, one on each monotone branch. Each such component is either
a point or a sliding segment. The proof of Lemma~\ref{lem:geometry}
also shows that the two intersections with the \(y\)-axis are
transverse, so no nonconstant sliding segment lies there.

Cutting the orbit at the switching lines and at its two intersections
with the \(y\)-axis therefore gives at most \(2m+2\) linear-flow arcs
and \(2m\) sliding arcs. At most one linear-flow arc in each half-plane
can contain the corresponding extreme value of \(x\); splitting there,
if necessary, gives at most \(2m+4\) linear-flow arcs. On each resulting
linear-flow arc, \(x\) is in fact strictly monotone. Indeed, otherwise
monotonicity would force \(x\) to be constant on a nontrivial
subinterval; since such an arc lies in \(x\neq0\), this would give
\(\dot x=0\) and \(y=F(x)\) constant, contradicting
\(\dot y=-x\neq0\).

For the combinatorial count, distinguish the increasing and decreasing
parts of the orbit in each half-plane. A linear-flow arc is labelled by
its linear region, the half-plane in which it lies, and the monotonicity
of \(x\), giving at most \(4(m+1)\) possible labels. A sliding arc is
labelled by its switching line and the part of the orbit on which it
occurs, giving at most \(2m\) further labels. Hence there are at most
\(6m+4\) possible labels. 

Each combinatorial type has length at most \(4m+4\). Hence the number \(W(m)\) of possible
combinatorial types satisfies
\[
W(m)
\leq \sum_{k=1}^{4m+4}(6m+4)^k
\leq [10(m+1)]^{4m+6}.
\]
This completed the proof.
\end{proof}

We next recall the standard Pfaffian terminology
\cite{Khovanskii} used to encode the individual arcs.
\begin{definition}
		Let \(U\subset\mathbb R^d\) be open, with
		\(\mathbf z=(z_1,\cdots,z_d)\). An ordered family of real analytic
		functions \((\varphi_1,\cdots,\varphi_\ell)\) on \(U\) is called a
		Pfaffian chain of length \(\ell\) and degree at most \(\alpha\) if
		\[
		\partial_{z_i}\varphi_j
		=
		P_{ij}(\mathbf z,\varphi_1,\cdots,\varphi_j),
		\]
		where each \(P_{ij}\) is a polynomial of total degree at most \(\alpha\). A function $f=P(\mathbf z,\varphi_1,\cdots,\varphi_\ell)$ is Pfaffian of degree at most $\beta$ with respect to the chain
		\((\varphi_1,\cdots,\varphi_\ell)\) if $P$ is a polynomial of total degree at most $\beta$. 
	\end{definition}

All chains below are defined on the whole space
\(\mathbb R^N\).
Once the system is fixed, its coefficients are constants
in the chain equations and are not counted as variables.

To apply the connected-component bound, we first express
the endpoint matching conditions for individual linear-flow
arcs as Pfaffian equations. Since \(F\) is affine on each
region determined by the switching lines, we consider
\begin{equation}\label{sys}
\dot x=y-ax-b,\qquad \dot y=-x,
\end{equation}
where \(a,b\in\mathbb R\) are fixed. We consider arcs on
which \(x\) is strictly monotone and whose images lie
in one of the closed half-planes \(x\geq0\) or \(x\leq0\).

Let \((X_0,Y_0)\) and \((X_1,Y_1)\) be the endpoints
of such an arc, with passage time \(\tau>0\), and put
\(\mathbf z_i=(X_i,Y_i-b)^{\mathsf T}\), \(i=0,1\). We write the
endpoint matching condition as the vector equation
\begin{equation}\label{eq:linear-matching}
\mathcal E_{a,b}(X_0,Y_0,X_1,Y_1,\tau)
:=\mathbf z_1-e^{\tau A}\mathbf z_0=\mathbf 0,
\qquad
A=\begin{pmatrix}-a&1\\-1&0\end{pmatrix}.
\end{equation}
Thus \eqref{eq:linear-matching} represents two scalar matching
equations. With \(a,b\) fixed, we seek a global Pfaffian representation
of these equations in the four endpoint coordinates
and a positive time parameter, with a common chain
depending only on that parameter. All degrees below
are total degrees in these variables and the chain functions.

The following lemma gives a uniform Pfaffian representation of the endpoint matching equations in \eqref{eq:linear-matching}.

\begin{lemma}\label{lem:flow-pfaffian}
The two scalar equations represented by \eqref{eq:linear-matching}
admit a Pfaffian representation satisfying:
\begin{enumerate}[
labelindent=0pt,
labelwidth=1.5em,
labelsep=0.4em,
leftmargin=1.9em,
align=left
]
\item[\textnormal{\textbf{(i)}}]
If \(a^2>4\), then
\(\ell\leq2\), \(\alpha\leq1\), and \(\beta\leq2\).
\item[\textnormal{\textbf{(ii)}}]
If \(a^2=4\), then
\(\ell\leq1\), \(\alpha\leq1\), and \(\beta\leq3\).
\item[\textnormal{\textbf{(iii)}}]
If \(a^2<4\), then
\(\ell\leq2\), \(\alpha\leq3\), and \(\beta\leq4\).
\end{enumerate}
The chosen time parameter is a local analytic coordinate for \(\tau\).
\end{lemma}

\begin{proof}
If \(a^2>4\), the entries of \(e^{\tau A}\) are linear
combinations of \(e^{\lambda_+\tau}\) and
\(e^{\lambda_-\tau}\), where
\(
\lambda_\pm=(-a\pm\sqrt{a^2-4})/2.
\)
These two exponential
functions form a global Pfaffian chain of length two and degree one.
Since the two scalar equations in \eqref{eq:linear-matching} are linear
in the endpoint variables and linear in the chain functions, their
function degree is at most two. This proves~\textbf{(i)}.

If \(a^2=4\), then
\[
e^{\tau A}
=
e^{-a\tau/2}
\left[I_2+\tau\left(A+\frac{a}{2}I_2\right)\right].
\]
The function \(e^{-a\tau/2}\) forms a global Pfaffian chain of length
one and degree one. The two scalar equations in
\eqref{eq:linear-matching} are polynomial in the endpoint variables,
\(\tau\), and this chain function, with total degree at most three,
thereby establishing~\textbf{(ii)}.

If \(a^2<4\), then
\[
e^{\tau A}
=e^{-a\tau/2}
\left(
\cos(\omega\tau)I_2+
\frac{\sin(\omega\tau)}{\omega}B
\right),
\]
where
\[
\omega=\sqrt{1-\frac{a^2}{4}},
\qquad
B=A+\frac{a}{2}I_2.
\]
Let \(x(t)\) be the horizontal coordinate along the arc,
\(0\leq t\leq\tau\). Since
\(
x''+ax'+x=0,
\)
the function \(x(t)\) is a nonzero sinusoid multiplied by
\(e^{-at/2}\), and consecutive zeros are separated by
\(\pi/\omega\). The half-plane condition and strict monotonicity imply
that \(x(t)\) has no interior zero and that its endpoint values are
distinct. Hence
\(
0<\omega\tau<\pi.
\)

Introduce
\begin{equation}\label{eq:half-angle}
u=\tan\frac{\omega\tau}{2}>0,\qquad
H(u)=\frac{1}{1+u^2},\qquad
E(u)=e^{-(a/\omega)\arctan u}.
\end{equation}
Since
\[
H'=-2uH^2,
\qquad
E'=-\frac{a}{\omega}HE,
\]
the pair \((H,E)\) forms a global Pfaffian chain of length two and
degree at most three. Thus
\(
\ell\leq2\) and \(\alpha\leq3.
\)
Using the half-angle identities and multiplying
\eqref{eq:linear-matching} by \(1+u^2\), we obtain the equivalent
matching equation
\begin{equation}\label{eq:focus-matching}
(1+u^2)\mathbf z_1
-
E(u)\left[
(1-u^2)I_2+\frac{2u}{\omega}B
\right]\mathbf z_0
=\mathbf 0.
\end{equation}
The two scalar equations represented by \eqref{eq:focus-matching} have
total degree at most four in the endpoint coordinates, the parameter
\(u\), and the chain functions.  The terms of
maximal degree are of the form
\(
u^2E(u)X_0
\) or \(
u^2E(u)Y_0.
\)
Hence
\(
\beta\leq4,
\) which completes ~\textbf{(iii)}.

Finally,
\[
\tau(u)=\frac{2}{\omega}\arctan u,
\qquad
\frac{\dd\tau}{\dd u}
=
\frac{2}{\omega(1+u^2)}>0,
\]
so \(u\) is a local analytic coordinate for \(\tau\). In cases
\textbf{(i)-(ii)}, the parameter is simply \(\tau>0\).
In  cases \textbf{(i)-(iii)}, the chain functions depend only on the time
parameter and extend globally to \(\mathbb R^5\) by independence of the
four endpoint coordinates.
\end{proof}

Fix a combinatorial type \(w\) arising from the decomposition in
Lemma~\ref{lem:arc-count}, and let \(k\) denote its total number of arcs.  By Lemma~\ref{lem:arc-count}, we have \(k\leq4m+4\).
Let
\[
P_j=(x_j,y_j),\qquad j=1,\cdots,k,\qquad P_{k+1}=P_1,
\]
be the cyclically ordered vertices, where \(P_1\) is the upper
intersection with the \(y\)-axis. The type also specifies the lower
axis vertex and the two consecutive half-plane blocks.

Associate a positive parameter \(u_j\) with the \(j\)-th arc. For an
arc with \(a_i^2<4\), \(u_j\) is the half-angle parameter in
\eqref{eq:half-angle}; for every other linear-flow or sliding arc, it
is the actual duration. Set
\[
\boldsymbol{\theta}_w
:=
(x_1,y_1,\cdots,x_k,y_k,u_1,\cdots,u_k)
\in\mathbb R^{3k}.
\]
These \(3k\) coordinates will be used below to define the set \(S_w\).

Consider the \(j\)-th linear-flow arc, assigned to the region
\((\xi_i,\xi_{i+1})\).  Put
\[
\tau_j=
\begin{cases}
 u_j, & a_i^2\geq4,\\[2mm]
 \displaystyle\frac{2}{\omega_i}\arctan u_j,
 & a_i^2<4,
\end{cases}
\qquad
\omega_i=\sqrt{1-\frac{a_i^2}{4}}
\quad (a_i^2<4).
\]
Its endpoint matching condition is
\begin{equation}\label{eq:arc-matching}
\mathcal E_j
:=\mathcal E_{a_i,b_i}
(x_j,y_j,x_{j+1},y_{j+1},\tau_j)
=\mathbf 0.
\end{equation}
Here and below, \eqref{eq:arc-matching} is understood through the
equivalent two-scalar-equation Pfaffian representation supplied by
Lemma~\ref{lem:flow-pfaffian}; in the focus case this is
\eqref{eq:focus-matching} with the corresponding subscripts.
Let \(d\in\{1,-1\}\) denote the prescribed direction of monotonicity.
Together with \eqref{eq:arc-matching}, impose
\begin{equation}\label{eq:valid-affine}
\begin{aligned}
&u_j>0,\qquad d(x_{j+1}-x_j)>0,\\
&d(y_j-a_i x_j-b_i)\geq0,\qquad
d(y_{j+1}-a_i x_{j+1}-b_i)\geq0,\\
&\xi_i\leq x_j\leq\xi_{i+1},\qquad
\xi_i\leq x_{j+1}\leq\xi_{i+1}.
\end{aligned}
\end{equation}
An inequality involving an infinite endpoint is omitted.
These are at most eight linear inequalities. The global
vertex conditions below place the two endpoints in the
assigned closed half-plane.

The following lemma ensures that the endpoint conditions select arcs remaining in their prescribed regions.
\begin{lemma}\label{lem:valid-affine}
Suppose that a linear-flow arc satisfies \eqref{eq:arc-matching} and
\eqref{eq:valid-affine}, and that its endpoints lie in the assigned
closed half-plane. Then \(x\) is strictly monotone in the prescribed
direction, and every interior point lies in the assigned open linear
region and open half-plane.
\end{lemma}
\begin{proof}
Let \(x(t)\), \(0\leq t\leq\tau\), be the horizontal component of
the linear-flow arc. By system \eqref{sys},
\(
\ddot x+a\dot x+x=0.
\)
We first show that \(x(t)\) has no zero in \((0,\tau)\).
In the focus case \(a^2<4\), one may write
\[
x(t)=e^{-at/2}
\bigl(C_1\cos(\omega t)+C_2\sin(\omega t)\bigr),
\qquad
\omega=\sqrt{1-\frac{a^2}{4}}.
\]
By the choice of the half-angle parameter,
\(0<\omega\tau<\pi\). Hence two consecutive zeros of a
nontrivial \(x(t)\) are separated by \(\pi/\omega>\tau\).
Since the endpoint abscissae have the same weak sign and are distinct,
there can be no interior zero.

If \(a^2>4\), then \(x(t)\) is a linear combination of two
distinct real exponentials and therefore has at most one zero. If
\(a^2=4\), then
\(
x(t)=e^{-at/2}(C_1+C_2t),
\)
so again it has at most one zero unless \(x\equiv0\), which is excluded
by the strict endpoint displacement. Thus the endpoint signs exclude
an interior zero also in the real-eigenvalue cases. Consequently, the
interior of the arc remains in the prescribed open half-plane.

It remains to prove strict monotonicity. At any interior critical point
\(t_*\), we have
\(
\dot x(t_*)=0\) and
\(\ddot x(t_*)=-x(t_*).
\)
Hence, in the right half-plane every critical point is a strict local
maximum, whereas in the left half-plane every critical point is a
strict local minimum. Equivalently, \(\dot x\) can change sign only
from positive to negative on the right, and only from negative to
positive on the left.

On the other hand, \eqref{eq:valid-affine} gives
\(
d\,\dot x(0)\geq0,\) 
\( d\,\dot x(\tau)\geq0\) and 
\( d\bigl(x(\tau)-x(0)\bigr)>0.
\)
These endpoint conditions are incompatible with an interior turning
point in the prescribed direction. Hence \(x\) is strictly monotone
along the whole arc.

Therefore, for \(0<t<\tau\),
\(
\min\{x(0),x(\tau)\}
<
x(t)
<
\max\{x(0),x(\tau)\}.
\)
Since both endpoints lie in the assigned closed linear region, every
interior point lies in the corresponding open linear region. Together
with the absence of interior zeros, this proves the assertion.
\end{proof}

It remains to incorporate sliding arcs, whose matching conditions are linear.
A sliding arc on a nonzero switching line \(x=\xi_i\) is described by
\begin{equation}\label{eq:sliding-conditions}
x_j=x_{j+1}=\xi_i,\qquad
u_j>0,\qquad
y_{j+1}=y_j-\xi_i u_j,
\end{equation}
together with the four linear inequalities placing
\(y_j\) and \(y_{j+1}\) in
\begin{equation}\label{eq:sliding-interval}
\left[
\min{F(\xi_i^-),F(\xi_i^+)},
\max{F(\xi_i^-),F(\xi_i^+)}
\right].
\end{equation}
This gives at most eight  conditions, at most five
of which are inequalities. At every intermediate point,
the vector
\(
(0,-\xi_i)
\)
belongs to the Filippov set~\eqref{eq:filippov}; hence
the entire segment is a valid Filippov solution. Both
attracting and repelling sliding segments are included.

Finally, set the abscissae of the two designated axis
vertices equal to zero, give each of the other \(k-2\)
vertices the strict sign of its assigned half-plane, and
require the upper axis vertex to lie above \(\beta_+\) and
the lower axis vertex to lie below \(\beta_-\).
The type is chosen so that every arc belongs to one of the
two consecutive half-plane blocks.  These global constraints contribute \(k+2\) conditions. 

Combining the conditions for linear-flow arcs, sliding arcs, and axis crossings yields the following uniform encoding.
\begin{proposition}\label{prop:encoding}
For each combinatorial type \(w\), the preceding conditions define a
set \(S_w\subset\R^{3k}\) with the following properties:
\begin{enumerate}[
labelindent=0pt,
labelwidth=1.5em,
labelsep=0.4em,
leftmargin=1.9em,
align=left
]
\item[\textnormal{\textbf{(i)}}] every point of \(S_w\) encodes a simple periodic Filippov orbit,
and every simple periodic orbit is encoded in at least one such set;
\item[\textnormal{\textbf{(ii)}}] \(S_w\) is defined by at most \(12k\) equations and inequalities.
All inequalities are linear in the coordinates of
\(\boldsymbol{\theta}_w\), and all equations are Pfaffian with respect
to a common global Pfaffian chain satisfying
\(
\ell\leq2k,\) \( \alpha\leq3\) and \(\beta\leq4.
\)
\end{enumerate}
\end{proposition}
\begin{proof}
Equation~\eqref{eq:arc-matching} together with
Lemma~\ref{lem:valid-affine} validates every linear-flow arc, whereas
\eqref{eq:sliding-conditions}--\eqref{eq:sliding-interval} validate
every sliding arc. The endpoints match cyclically. After recovering
the actual durations of the focus arcs, their concatenation is locally
absolutely continuous and satisfies the Filippov inclusion almost
everywhere. It is nonconstant because the two designated axis heights
are strictly separated.

On the right half-plane block, \(y\) is strictly decreasing; on the
left block, it is strictly increasing. Each block is therefore
injective. Their interiors lie in different open half-planes, and
their only common points are the two axis vertices. Hence the resulting
closed orbit is simple. Conversely, the decomposition in
Lemma~\ref{lem:arc-count} satisfies exactly the conditions above, so
every simple periodic orbit is represented by at least one set \(S_w\).
This proves~\textbf{(i)}.

We next verify~\textup{(ii)}. Each linear-flow arc contributes at most ten conditions, each sliding
arc at most eight, and the global vertex conditions contribute another
\(k+2\). Since \(k\geq2\), the total number
of  conditions is at most
\[
10k+(k+2)=11k+2\leq12k.
\]
By construction, all inequalities occurring in these conditions are linear in the coordinates of \(\boldsymbol{\theta}_w\).

For the Pfaffian equations, concatenate the chains supplied by
Lemma~\ref{lem:flow-pfaffian}, using the separate parameter of each
linear-flow arc. If the \(j\)-th arc contributes a chain of length
\(\ell_j\leq2\), then the resulting common chain has length
\[
\ell=\sum_j\ell_j\leq2k.
\]
Every chain function depends only on the parameter of its own arc, so
its derivatives with respect to the variables belonging to the other
arcs vanish. Hence the concatenated chain remains triangular. By
Lemma~\ref{lem:flow-pfaffian}, its chain degree satisfies
\(
\alpha\leq3.
\)

By Lemma~\ref{lem:flow-pfaffian}, the two scalar equations represented
by \eqref{eq:arc-matching} have function degree at most four for every
linear-flow arc. The equations describing the sliding arcs and the
global vertex conditions have smaller degree.
Therefore all equations defining \(S_w\) are Pfaffian with respect to
this common global chain and satisfy
\(
\beta\leq4.
\)
This completes the proof of ~\textbf{(ii)}.
\end{proof}

We next pass from the Pfaffian encoding to a topological count.
Let \(b_0(S)\) denote the number of connected components of a set \(S\).
The following lemma relates the number of limit cycles represented by
\(S_w\) to \(b_0(S_w)\).
\begin{lemma}\label{lem:cycles-components}
For a fixed system and a fixed type \(w\), the number
of limit cycles represented by \(S_w\) is at most
\(b_0(S_w)\).
\end{lemma}
\begin{proof}
For $q\in S_w$, let $K(q)$ denote the compact
geometric image of the periodic orbit encoded
by $q$, and let $N_w(q)$ be the union of the
images of its closed linear-flow arcs, including
their endpoints. Thus $N_w(q)$ is precisely
the nonsliding part of $K(q)$.

Parametrize each arc on a fixed unit interval.
The explicit linear flows depend continuously
on their endpoints and durations. In the focus
case, $\tau=2\arctan(u)/\omega$ depends
continuously on $u$, while a sliding arc depends
continuously on its endpoints. Since the
combinatorial type $w$ is fixed, both
\(
q\mapsto K(q)\)
and \(
q\mapsto N_w(q)
\)
are continuous in the Hausdorff metric.

Let $\gamma$ be a limit cycle represented by
$S_w$, set
\[
F_\gamma
=\{q\in S_w:K(q)=\gamma\},
\] and
fix $q_0\in F_\gamma$.
If the isolation condition concerns the whole
periodic orbit, there exists an open neighborhood
$U$ of $\gamma$ containing no other complete
periodic orbit. By the Hausdorff continuity of
$K$, we have $K(q)\subset U$ for all
$q\in S_w$ sufficiently close to $q_0$.
Hence $K(q)=\gamma$.

If the isolation condition concerns the
nonsliding part, let $U$ be the corresponding
isolating neighborhood of $\gamma_{\mathrm{ns}}$.
Since
\(
N_w(q_0)=\gamma_{\mathrm{ns}},
\)
the Hausdorff continuity of $N_w$ implies
that $N_w(q)\subset U$ for all $q\in S_w$
sufficiently close to $q_0$.
The fixed type preserves any prescribed
nonconstant sliding arcs, so the corresponding
isolation condition remains applicable.
Hence $K(q)=\gamma$, proving that
$F_\gamma$ is open in $S_w$.
Moreover, $F_\gamma$ is closed in $S_w$
by the Hausdorff continuity of $K$.  Consequently, every connected component
of $S_w$ meeting $F_\gamma$ is contained
in $F_\gamma$.

Therefore, no connected component of $S_w$
can represent two distinct limit cycles,
and the number of limit cycles represented
by $S_w$ is at most $b_0(S_w)$.
\end{proof}
To estimate \(b_0(S_w)\), we replace the inequalities defining
\(S_w\) by equations in an enlarged space. 
By Proposition~\ref{prop:encoding}, every such inequality is of the form
\(g(\boldsymbol{\theta}_w)\geq0\) 
or \(g(\boldsymbol{\theta}_w)>0,\)
where \(g\) is linear in the coordinates of \(\boldsymbol{\theta}_w\).

For each non-strict inequality \(g\geq0\), introduce a new real
variable \(s\) and impose
\(
g(\boldsymbol{\theta}_w)-s^2=0.
\)
This equation is solvable in \(s\) precisely when \(g\geq0\).
Similarly, for each strict inequality \(g>0\), introduce a new real
variable \(s\) and impose
\(
g(\boldsymbol{\theta}_w)s^2-1=0,
\)
which is solvable precisely when \(g>0\). Since \(g\) is linear, all
equations introduced in this way have degree at most three.

Let \(V_w\) denote the common zero set of the equations defining
\(S_w\), together with all the auxiliary equations introduced above.
The Pfaffian chain from Proposition~\ref{prop:encoding} is extended
independently of the new variables. The coordinate projection
\[
\pi_w:V_w\longrightarrow S_w,\qquad
\pi_w(\boldsymbol{\theta}_w,\mathbf s)
=\boldsymbol{\theta}_w,
\]
is continuous and surjective.
Since the image of every connected component of \(V_w\) under
\(\pi_w\) is connected, we have
\begin{equation}\label{eq:projection}
b_0(S_w)\leq b_0(V_w).
\end{equation}
It remains to estimate the Pfaffian format of \(V_w\).
\begin{lemma}\label{lem:format}
For each combinatorial type \(w\), let
\(V_w\subset\mathbb R^N.\) Then \(V_w\) is defined with respect to a common
global Pfaffian chain satisfying
\[
N\leq15k,\qquad
\ell\leq2k,\qquad
\alpha\leq3,\qquad
\beta\leq4.
\]
Moreover, since \(k\leq4(m+1)\),
\begin{equation}\label{eq:format}
N\leq60(m+1),\qquad
\ell\leq8(m+1).
\end{equation}
\end{lemma}
\begin{proof}
Since
\(
\boldsymbol{\theta}_w\in\mathbb R^{3k},
\)
the parameter space of \(S_w\) has dimension \(3k\).
By Proposition~\ref{prop:encoding}, the set \(S_w\) is defined by at
most \(12k\)  equations and inequalities. Since one auxiliary
real variable is introduced for each inequality, the total number of variables
in \(V_w\) satisfies
\[
N\leq 3k+12k=15k.
\]

The common Pfaffian chain associated with \(S_w\) is extended
independently of the auxiliary variables, so neither its length nor
its chain degree changes. Hence Proposition~\ref{prop:encoding} gives
\(
\ell\leq2k \) and \(\alpha\leq3.\) 
By the preceding construction, each auxiliary equation has degree at
most three and introduces no new Pfaffian chain function. Since the  equations defining \(S_w\) have function degree at most four,
all equations defining \(V_w\) have function degree at most four.
Thus
\(
\beta\leq4.
\)

Finally, Lemma~\ref{lem:arc-count} gives
\(
k\leq4(m+1).
\)
Therefore
\[
N\leq15k\leq60(m+1),
\qquad
\ell\leq2k\leq8(m+1),
\]
which proves \eqref{eq:format}.
\end{proof}
We are now ready to combine the component estimate with the uniform
bounds on the number and complexity of the combinatorial types.

\noindent{\bf Proof of Theorem~\ref{thm:upper}.}
Let \(F\in\Pm\) be arbitrary.
By Lemmas~\ref{lem:geometry} and~\ref{lem:arc-count}, every limit cycle has one of the combinatorial types considered above, and by Proposition~\ref{prop:encoding}, it is represented by the corresponding set \(S_w\).

 For each combinatorial type \(w\),
Lemma~\ref{lem:cycles-components} and \eqref{eq:projection} show that
the number of limit cycles represented by \(S_w\) is at most
\(b_0(V_w)\).
Since the common Pfaffian chain is defined
on the whole space \(\mathbb R^N\),
\cite[Corollary~3.3]{GV} gives
\[
b_0(V_w)
\leq
2^{\ell(\ell-1)/2+1}
\beta(\alpha+2\beta-1)^{N-1}
\bigl((2N-1)(\alpha+\beta)-2N+2\bigr)^\ell.
\]
It follows from Lemma \ref{lem:format} that
\[
\begin{aligned}
b_0(V_w)
&\leq
8\cdot 2^{\ell(\ell-1)/2}
10^{N-1}(12N-5)^\ell\leq
8\cdot 2^{32(m+1)^2-4(m+1)}
10^{64(m+1)-1}
[768(m+1)]^{8(m+1)}.
\end{aligned}
\]

Since, by Lemma~\ref{lem:arc-count}, the number of possible types is at
most
\(
[10(m+1)]^{4m+6},
\)
the total number of limit cycles of system~\eqref{eq:system} for this
arbitrary \(F\in\Pm\) is at most
\[
\begin{aligned}
B(m)
={}&
8[10(m+1)]^{4m+6}
2^{32(m+1)^2-4(m+1)}
\times
10^{64(m+1)-1}
[768(m+1)]^{8(m+1)}.
\end{aligned}
\]
Since this bound is independent of \(F\), taking the supremum over
\(F\in\Pm\) yields
\(
H_m\leq B(m).
\)
It remains to derive the simpler estimate. 

Set \(q=m+1\geq2\).
Using the elementary bounds
\(
8\leq2^{q^2},\)
\( 10q\leq2^{3q},\)
\(768q\leq2^{6q}\) and
\(10<2^4,
\)
together with \(4m+6=4q+2\leq5q\), we obtain
\[
\begin{aligned}
B(m)
<
2^{q^2}
2^{15q^2}
2^{32q^2-4q}
2^{256q}
2^{48q^2}=
2^{96q^2+252q}
<
2^{96q^2+256q}.
\end{aligned}
\]
Since \(q\geq2\), \(256q\leq128q^2\), and hence
\[
H_m\leq B(m)
<
2^{224q^2}
=
2^{224(m+1)^2}.
\]
This proves the theorem.
\hfill\(\Box\)
     
	\section*{Acknowledgements}
All authors are supported by the  National Natural Science Foundation of China (Nos. 12571190, 12322109) and Hunan Basic Science Research Center for Mathematical Analysis (2024JC2002).

\end{document}